\documentclass[12pt]{article}
\usepackage{latexsym}
\usepackage{amsmath}
\usepackage{amssymb}
\usepackage{amsthm}
\usepackage{mathtools}
\usepackage{dsfont}
\usepackage{bbm}
\usepackage{color}
\usepackage[tmargin=1.25in,bmargin=1.25in,lmargin=1.1in,rmargin=1.1in]{geometry}

\usepackage{graphicx}
\usepackage{epsfig}
\usepackage{hyperref} 

\definecolor{Red}{rgb}{1,0,0}
\definecolor{halfgray}{gray}{0.55}
\definecolor{webgreen}{rgb}{0,.5,0}
\definecolor{webbrown}{rgb}{.6,0,0}
\definecolor{Maroon}{cmyk}{0, 0.87, 0.68, 0.32}
\definecolor{RoyalBlue}{cmyk}{1, 0.50, 0, 0}
\definecolor{Black}{cmyk}{0, 0, 0, 0}

\hypersetup{
    colorlinks=true, linktocpage=true, pdfstartpage=1, pdfstartview=FitV,
    breaklinks=true, pdfpagemode=UseNone, pageanchor=true, pdfpagemode=UseOutlines,
    plainpages=false, bookmarksnumbered, bookmarksopen=true, bookmarksopenlevel=1,
    hypertexnames=true, pdfhighlight=/O,
    urlcolor=webbrown, linkcolor=RoyalBlue, citecolor=webgreen,
}

\newtheorem{thm}{Theorem}[section]

\newtheorem{lem}[thm]{Lemma}

\newtheorem{remark}[thm]{Remark}

\makeatletter\@addtoreset{equation}{section}\makeatother

\begin{document}

\newcommand{\D}{\mathrm{d}}
\newcommand{\R}{\mathbb{R}}
\newcommand{\EE}{\mathbb{E}}
\newcommand{\PP}{\mathbb{P}}
\newcommand{\var}{\operatorname{Var}}
\newcommand{\cov}{\operatorname{Cov}}
\newcommand{\G}{\mathcal{G}}
\newcommand{\B}{\mathcal{B}}
\newcommand{\e}{\mathrm{e}}
\newcommand{\p}{\mathrm{P}}
\newcommand{\N}{\mathbb{N}}
\newcommand{\K}{\mathbb{K}}
\newcommand{\T}{\mathbb{T}}
\newcommand{\bbP}{\mathbb{P}}
\newcommand{\bbE}{\mathbb{E}}
\newcommand{\bfP}{\mathbf{P}}
\newcommand{\bfE}{\mathbf{E}}
\newcommand{\calE}{\mathcal E}
\newcommand{\GW}{\operatorname{GW}}
\newcommand{\Tra}{\operatorname{Tra}}
\newcommand{\RW}{\operatorname{RW}}
\newcommand{\RWP}{\RW\times\bfP}
\newcommand{\F}{\mathcal F}
\newcommand{\esc}{\operatorname{esc}}

\pagenumbering{arabic}

\newcommand{\Peaug}{\mathrm{P}_\varepsilon^{\mathrm{aug}}}
\newcommand{\Pehat}{\hat{\mathrm{P}}_\varepsilon}
\newcommand{\PPehat}{\hat{\mathbb{P}}_\varepsilon}
\newcommand{\Eehat}{\hat{\mathrm{E}}_\varepsilon}
\newcommand{\EEehat}{\hat{\mathbb{E}}_\varepsilon}

\title{Biased random walk on a locally infinite Galton-Watson tree}
\author{Jan Nagel} 
\maketitle

\begin{abstract} 
We propose a new variant of a biased random walk on a Galton-Watson tree, in which some vertices have infinitely many children. These vertices induce a natural regeneration structure, which allows to study differentiability and also monotonicity of the speed of the random walk in the bias parameter. In particular, we give explicit examples of offspring distributions for which the speed is strictly monotone increasing. Additionally, we show that the speed in this model is the limit of the speed of random walks on appropriate locally finite trees. 
\end{abstract}


\smallskip


\section{Introduction}
\label{sec:intro}

We consider a biased random walk on a Galton-Watson tree, where some vertices have infinitely many direct offspring, and which is constructed as follows. Let $\nu$ be a probability measure on $\mathbb{N}_0$ with finite mean $m(\nu)$ and $\alpha\in (0,1)$. Starting with the root, the number of children of a vertex $v$ is sampled from $\nu$ with probability $1-\alpha$, and with probability $\alpha$, attach to $v$ countably infinite many children. The vertices with infinitely many children will be referred to as marked vertices. We write $d(v)$ for the number of children of a vertex $v$ and $v^*$ for the parent vertex of $v$. 

Given a realization $T$ of the tree conditioned on survival and a bias parameter $\beta\geq 0$, the biased random walk $(X_n)_{n\geq 0}$ is the process starting at the root $\rho$, and when $v\neq \rho$ is an unmarked vertex, the random walk moves to a neighboring vertex with probabilities 
\begin{align} \label{eq:transitionprob}
P_{T,\beta}(X_{n+1}=v^* \mid X_n=v) = \frac{1}{1+d(v)\beta} , \qquad P_{T,\beta}(X_{n+1}=w \mid X_n=v) = \frac{\beta}{1+d(v)\beta} ,
\end{align}
for $w$ a descendant of $v=w^*$.  When the random walk is at a marked vertex with infinitely many descendants, it moves to the first unvisited descendant. When at $\rho$ and the root is not marked, the random walk always moves to a uniformly chosen descendant. This process is Markovian only up to the hitting time of a marked vertex, but it is a Markov process when we keep track of the visited children of the last visited marked vertex.

The biased random walk with transition probabilities \eqref{eq:transitionprob} is a classical model for a random walk in random environment. It was introduced by \cite{Lyons1990} and a law of large numbers was proven by \cite{LyoPemPer96a,LyoPemPer96b}, with limit the effective velocity or speed $v(\beta)$ of the random walk. Except for $\beta=1$ or when the offspring distribution is concentrated on a single integer, there is no closed-form expression for the speed, despite an elegant formula for $v$ proven by \cite{aidekon2014biased}. Central limit theorems for the biased random walk were considered by \cite{PerZei2008} and large deviation principles by \cite{Dembo2002largeDeviations}. Several authors have studied properties of the speed $v$ as a function of the bias parameter. Differentiability at the critical value for transience was proven by \cite{benarous2013einstein}, and by \cite{bowditch2020differentiability} in the parameter range where the time spend in finite subtrees has finite second moment. The effects of such trapping phenomenon was studied in detail by \cite{arous2012biased}, and on conditioned critical trees by \cite{CroydonFriberghKumagai2013CriticalTrees}. For trees without leaves, it was conjectured in \cite{LyoPemPer96a}, that $v$ is increasing in the bias parameter. For such trees, \cite{benarous2014monotonicity} show that the speed is eventually increasing with a coupling argument, and A\"id\'ekon gives an analytic argument for monotonicity on $[2,\infty)$, see the survey article \cite{BenArousFribergh2016Survey}. So far, monotonicity of $v$ on the full parameter range is still unproven.

We are interested in the speed of the random walk in the new model. Visiting the marked vertices leads to a regeneration event: the random walk visits a new subtree it has never seen before. This makes it easy to prove, for example, the existence of an asymptotic speed. We also study the differentiability of the speed and prove monotonicity of the speed for specific offspring distributions. 

In the following, we write $P_{T,\beta}$ for the quenched law of the biased random walk on a tree $T$ starting at the root, $\mathrm{P}$ for the law of the tree conditioned on survival and $\bbP_\beta$ for the annealed law, when we average over $P_{T,\beta}$ with respect to $\mathrm{P}$.  We denote by $|v|$ the graph distance of $v\in T$ to the root, by $f$ the generating function of the offspring law $\nu$ and write $p_k$ for $\nu(\{k\})$. In particular, $p_0$ is the probability that an unmarked vertex is a leaf. 
The first result concerns the existence of the speed.  

\begin{thm} \label{thm:speedinfinite}
For all $\beta> 0$ there exists a constant $v_\infty(\beta)$, such that $\PP_\beta$-almost surely 
\begin{align*}
\lim_{n\to \infty} \frac{|X_n|}{n} = v_\infty(\beta) .
\end{align*}
If $p_0=0$ or $p_0=1$, we have $v_\infty(\beta)>0$ for all $\beta> 0$. If $p_0\in(0,1)$,  $v_\infty(\beta)>0$ if and only if $\beta <\beta_c=((1-\alpha)f'(q_\infty))^{-1}$, where $q_\infty$ is the unique solution to $q_\infty = (1-\alpha)f(q_\infty)$.  
\end{thm}

In the extreme case $\beta=0$, the random walk always moves back towards the root, except when it is at a marked vertex. When we condition on the root being a marked vertex, we still have a law of large numbers for $\beta=0$ with explicit value $v_\infty(0)=\frac{\alpha}{2-\alpha}$, see Section \ref{sec:infinitespeed}. The same behavior holds when $\nu=\delta_0$, in which case $p_0=1$ and conditioning on an infinite tree guarantees a marked root. Then the speed equals $\frac{\alpha}{2-\alpha}$ for all $\beta\geq 0$. 

The tree with potentially infinitely many direct descendants can be approximated in a natural way by classical Galton-Watson trees, when we attach to a marked vertex only $N$ descendant, for $N\in \mathbb N$. This yields a Galton-Watson tree with offspring law
\begin{align*}
\nu_N = \alpha \delta_N + (1-\alpha)\nu . 
\end{align*}
that is, with probability $(1-\alpha)$, a vertex has a random number of offspring distributed according to the law $\nu$, and with probability $\alpha$, the vertex has $N$ offspring. Let $m(\nu)$ denote the mean of the offspring law $\nu$, so that 
\begin{align*}
m_N = m(\nu_N) = \alpha N + (1-\alpha) m(\nu) 
\end{align*} 
is the mean of $\nu_N$. Assume $N$ is sufficiently large, such that $m_N>1$ and the tree is supercritical. We write $\mathrm{P}^{(N)}$ for the law of such a Galton-Watson tree, conditioned on survival, and $\bbP_\beta^{(N)}$ for the annealed law of the random walk with transition probabilities \eqref{eq:transitionprob}. 

The law of large numbers by \cite{LyoPemPer96b} shows that for all $\beta\geq 0$ there exists a deterministic constant $v_N(\beta)$, such that $\mathbb P^{(N)}_{\beta}$-almost surely
\begin{align*}
\lim_{n\to \infty} \frac{|X_n|}{n} = v_N(\beta) .
\end{align*}
We can show that $v_N$ converges to the speed in the new model with infinitely many children. For this we assume that the tree has no leaves, $p_0=0$. In this case, the speed $v_N(\beta)$ is positive if and only if $\beta>1/m_N$. 
%

\begin{thm} \label{thm:speedconvergence}
When $p_0=0$, we have for all $\beta>0$
\begin{align*}
\lim_{N\to \infty} v_N(\beta) = v_\infty(\beta). 
\end{align*}
\end{thm}

Since $v_N(0)=0$ for all finite $N$, but $v_\infty(0)=\alpha/(2-\alpha)$, we do not have convergence of the speed in 0. 
The main difficulty in proving the convergence of the speed is to find moment bounds for regeneration times which hold uniformly in $N$. This is the reason for the condition $p_0=0$, since when $p_0>0$, we need to condition the tree on survival. Under this conditioning, there is in general no increasing coupling between trees corresponding to different $\nu_N$. In particular, finite subtrees attached to a backbone tree may increase in size when $N$ is increased. 

We complement the convergence of the speed with a convergence of its derivative, where we need some additional assumptions on the offspring law $\nu$. In Section \ref{sec:infinitespeed}, we show differentiability of $v_\infty$ when $p_0=0$ and $\nu$ has a finite second moment. 
\cite{bowditch2020differentiability} prove differentiablility of the speed $v_N$ under an exponential moment assumption on $\nu$. 

\begin{thm} \label{thm:derivativeconvergence}
If $p_0=0$, and $\sum_{k=1}^\infty p_k\gamma^k<\infty$ for some $\gamma>1$, then we have for all $\beta >0$ 
\begin{align*}
\lim_{N\to \infty} v_N'(\beta) = v_\infty'(\beta). 
\end{align*}
\end{thm}


Our final result in this section concerns the monotonicity of $v_\infty$ in the bias parameter. When all unmarked vertices have the same degree, we can give an explicit criterion for the density of marked vertices, which leads to monotonicity on the full parameter range.

\begin{thm} \label{thm:speedmonotone}
If $\nu = \delta_d$ for $d\in \mathbb N$ and 
\begin{align*}
\alpha >1-\frac{1}{72\mathrm{e} d} ,
\end{align*} 
then $v_\infty$ is strictly monotone increasing on $(0,\infty)$. 
\end{thm}

The remainder of this paper is structured as follows. In Section \ref{sec:infinitespeed}, we give more details on the speed in our new model and also prove its differentiability. Section \ref{sec:convergence} is devoted to the proof of the convergence in Theorem \ref{thm:speedconvergence} and Section \ref{sec:diffconvergence} proves the convergence of $v_N'$ in Theorem \ref{thm:derivativeconvergence}. The monotonicty in Theorem \ref{thm:speedmonotone} is proven in Section \ref{sec:monotonicity}. The proofs of the main theorems rely on moment bounds for different kinds of regeneration times, uniform in the bias or in $N$, which are collected and proven in Section \ref{sec:moments}.

\section{On the speed with infinitely many children}
\label{sec:infinitespeed}

We write $M$ for the set of all marked vertices. Define the hitting time of the set $M$ as
\begin{align} \label{eq:hittingsigma}
\sigma = \sigma_1= \inf\{n\in \mathbb N : X_n \in M \}. 
\end{align}
This time is finite $\bbP_\beta$-almost surely for any $\beta>0$, since with probability 1, the random walk visits infinitely many distinct vertices, and each newly visited vertex has probability $\alpha$ to be a marked vertex.  
We set recursively for $k\geq 1$
\begin{align} \label{eq:hittingsigma2}
\sigma_{k+1} = \inf\{n>\sigma_k: X_n \in M \} . 
\end{align}
At each of the times $\sigma_k$, the random walk starts by construction an excursion into a new subtree it has never seen before. Therefore, the increments between hitting times and generations define a sequence
\begin{align*}
\big( |X_{\sigma_{k+1}}| - |X_{\sigma_k}|, \sigma_{k+1}-\sigma_k\big)_{k\geq 1}
\end{align*}
of random variables, which are independent and identically distributed under $\bbP_\beta$. The increments until time $\sigma_1$ have a different distribution, since the root is not guaranteed to be a marked vertex, whereas $X_{\sigma_k}\in M$ by construction for any $k\geq 1$. Define a new environment law $\mathrm{P}^*$ under which we condition the root to be marked, the corresponding annealed law is then denoted by $\bbP^*_\beta$. Note that when the root is marked the tree is infinite a.s., so that this conditioning becomes unnecessary. We then have that the distribution of $( |X_{\sigma_{k+1}}| - |X_{\sigma_k}|, \sigma_{k+1}-\sigma_k )$ under $\bbP_\beta$ is equal to the distribution of $( |X_{\sigma_1}|, \sigma_{1} )=( |X_\sigma|, \sigma ) $ under $\bbP_\beta^*$. 
If $\bbE_\beta[\sigma_2-\sigma_1]=\bbE_\beta^*[\sigma] = \infty$, we have almost surely under $\bbP_\beta$ and under $\bbP_\beta^*$
\begin{align*}
v_\infty(\beta) := \lim_{n\to \infty} \frac{|X_n|}{n} = 0  
\end{align*}
as in \cite[Proposition 4.1]{LyoPemPer96b}. 
If
\begin{align*}
\bbE_\beta[\sigma_2-\sigma_1] = \bbE_\beta^*[\sigma] < \infty ,
\end{align*}
the following convergence holds $\bbP_\beta$-almost surely and also $\bbP_\beta^*$-almost surely:
\begin{align} \label{eq:speedformula}
v_\infty(\beta) := \lim_{n\to \infty} \frac{|X_n|}{n} = \frac{\bbE_\beta[|X_{\sigma_{2}}| - |X_{\sigma_1}|]}{\bbE_\beta[\sigma_2-\sigma_1]} = \frac{\bbE_{\beta}^*[|X_\sigma|]}{\bbE_{\beta}^*[\sigma]} ,
\end{align}
see the proof of Theorem 3.1 in \cite{LyoPemPer96b}. This proves the almost sure convergence in Theorem \ref{thm:speedinfinite}.

Let us remark that if $\beta=0$ and the root is marked, $\sigma=1$ if the first visited vertex after the root is a marked vertex, and $\sigma=2$ otherwise. In the first case (which occurs with probability $\alpha$), we have $|X_\sigma|=1$, in the second case we have $|X_\sigma|=0$. This implies that $\bbP_0^*$-almost surely
\begin{align} \label{eq:speedformulain0}
v_\infty(0):=  \lim_{n\to \infty} \frac{|X_n|}{n} = \frac{\bbE_{0}^*[|X_\sigma|]}{\bbE_{0}^*[\sigma]}= \frac{\alpha}{2-\alpha} ,
\end{align}  
which extends $v_\infty$ to $[0,\infty)$.  
The following lemma completes the proof of Theorem \ref{thm:speedinfinite}.

\begin{lem}\label{lem:criticalvalue}
If $p_0\in \{0,1\}$, then $\bbE_{\beta}^*[\sigma]<\infty$ for all $\beta\geq 0$. If $p_0\in (0,1)$, then $\bbE_{\beta}^*[\sigma]<\infty$ if and only if $\beta <\beta_c=((1-\alpha)f'(q_\infty))^{-1}$, where $q_\infty$ is the unique solution to $q_\infty = (1-\alpha)f(q_\infty)$. 
\end{lem}

\begin{proof} 
For $p_0=1$, every visit to an unmarked vertex is directly followed by returning to the parent vertex and $\sigma$ may be dominated by a geometric random variable, multiplied by 2. If $p_0=0$, we may use the estimate in Lemma \ref{lem:moment-uniform2nd}, which implies that $\bbE_{\beta}[\sigma]$ is finite. 

Now suppose $p_0\in (0,1)$. We let $T_0$ be the subtree of one descendant $v$ (the first, say) of the root, with the additional edge connecting it to $\rho$. Let $q_\infty$ be the extinction probability of $T_0$. The probability that $T_0$ is finite is the unique fixed point in $(0,1)$ of $(1-\alpha)f$, by the usual extinction argument for Galton-Watson trees, with the offspring probability $p_k$ replaced by $(1-\alpha)p_k$. Also, note that we do not have to specify choosing the smallest fixed point, since 1 is not a fixed point of $(1-\alpha)f$, so that $q_\infty$ is uniquely determined by $q_\infty = (1-\alpha)f(q_\infty)$. \cite[Theorem 4.1]{LyoPemPer96b} shows that  $\bbE_{\beta}[\sigma]<\infty$ if and only if the mean return time to $\rho$ of a random walk in $T_0$ is finite, averaged over all realizations of $T_0$ conditioned to be finite. The tree $T_0$ is finite if and only if $v$ is not marked and all subtrees at children of $v$ are finite. Therefore, the generating function of the offspring law of $T_0$, conditioned on $T_0$ being finite, is given by
\begin{align*}
h(s) = \frac{(1-\alpha)f(q_\infty s)}{q_\infty} .
\end{align*} 
Mean hitting times can be computed in terms of the conductances of the edges, we refer to \cite{doyle1984random} and \cite{LyoPer16} for an overview on such electrical network arguments. For the random walk with transition probabilities \eqref{eq:transitionprob}, the conductance of the edge between generation $n$ and generation $n+1$ is given by $\beta^n$. The mean return time to $\rho$ is given by the inverse of the stationary measure evaluated at $\rho$, which may be computed by summing over all vertices the sum of adjacent edge weights. Writing $G_i(v)$ for the number of descendants of $v$ at distance $i$, it is given by 
\begin{align} \label{eq:hittingtimeformula}
2+2 \sum_{i=1}^\infty \beta^{i} \mathrm{E}[ G_i(v) ] & = 2+2 \sum_{i=1}^\infty  \beta^{i} h'(1)^{i} \\ \notag 
& = 2+2 \sum_{i=1}^\infty (\beta(1-\alpha)f'(q_\infty))^{i} .
\end{align}
This sum is finite if and only if $\beta<1/(1-\alpha)f'(q_\infty)$. 
\end{proof}

Next, we give a formula for the derivative of $v_\infty$ on $[0,\infty)$, when $v_\infty(0)$ is defined by \eqref{eq:speedformulain0}. In the following, we write $P_{T,\beta}^v$ for the law of a biased random walk starting at $v$ and $E_{T,\beta}^v$ for the corresponding expectation. We denote by 
\begin{align*}
\Delta_\beta(v)=E_{T,\beta}^v[|X_1|-|v|] 
\end{align*} 
the local bias in vertex $v$. If $v$ is a marked vertex, $\Delta_\beta(v)=1$ for all $\beta\geq 0$.

\begin{thm}\label{thm:infinitespeeddiff}
Suppose $p_0=0$ and the second moment of $\nu$ is finite. Then the speed $v_\infty$ is differentiable on $(0,\infty)$ and its derivative is given by 
\begin{align*}
v_\infty'(\beta) = \frac{1}{2\beta \bbE_\beta^*[\sigma]} \bbE_\beta^*\left[ \left( |X_\sigma| - \sigma v_\infty(\beta)\right)\left( |X_\sigma| - \sum_{k=0}^{\sigma-1} \Delta_\beta(X_k) \right)\right] ,
\end{align*}
for $\beta>0$. The right-hand derivative in 0 is given by
\begin{align*}
v_\infty'(0) = 2 \frac{\alpha (1-\alpha)}{(2-\alpha)^2} m(\nu) .
\end{align*}
\end{thm}

\begin{remark}\label{rem:einstein}
One may ask about the validity of an ``Einstein relation'' type of equation in $\beta=0$, which states that the mobility of a particle, i.e., the derivative of $v_\infty$ in 0, equals the diffusivity $\sigma^2_\infty(0)$, i.e., the asymptotic variance in a CLT under $\bbP_0^*$. In our new model however, $\beta=0$ is not a classical equilibrium point, the speed $v_\infty(0)$ is positive and $v_\infty'(0)$ does not agree with the standard notion of mobility.   
To compare the diffusivity with the derivative at zero, we use that under $\bbP_0^*$, the hitting time $\sigma$ is 1 or 2 and the variance $\sigma^2_\infty(0)$ may be computed explicitly as
\begin{align*}
\sigma^2_\infty(0)= \frac{1}{\bbE_0^*[\sigma]} \bbE_0^*\left[ \left( |X_\sigma| - \sigma v_\infty(\beta)\right)^2\right] = 4\frac{\alpha(1-\alpha)}{(2-\alpha)^3} .
\end{align*}
In particular, it is independent of $m(\nu)$ and in general does not agree with the derivative of $v_\infty$ at zero. 
\end{remark}

\noindent\textit{Proof of Theorem \ref{thm:infinitespeeddiff}.}
We differentiate numerator and denominator of the speed formula in \eqref{eq:speedformula} and first consider $\beta>0$. The numerator is 
\begin{align*}
\bbE^*_\beta[|X_\sigma|] = \mathrm{E}^*\big[ E_{T,\beta}[|X_\sigma|]\big]  = \mathrm{E}^*\left[ \sum_{m=1}^\infty E_{T,\beta}[|X_\sigma|\mathbbm{1}_{\{\sigma=m\}}] \right] . 
\end{align*} 
For a fixed tree $T$, the inner expectation can be written as 
\begin{align} \label{eq:diffinner}
E_{T,\beta}[|X_\sigma|\mathbbm{1}_{\{\sigma=m\}}] 
= \sum_z |z_m| P_{T,\beta}(X_0=z_0, X_1=z_1,\dots , X_m=z_m) , 
\end{align} 
where the sum is over all nearest neighbor paths $z=(z_0,\dots ,z_m)$ in $T$ with $z_0=\rho$, $z_1$ is the first descendant of the marked vertex $z_0$, and $z_m$ is the first marked vertex along this path after $z_0$. Write $p_{T,\beta}(v,w)$ for the transition probability from vertex $v$ to vertex $w$. If $v\neq \rho$ and $w$ is a descendant of $v$,
\begin{align}
\frac{\mathrm d}{\mathrm  d\beta} p_{T,\beta}(v,w) & = \frac{\mathrm d}{\mathrm d\beta} \frac{\beta}{d(v)\beta +1} = \frac{1}{(d(v)\beta+1)^2}
= \frac{1}{\beta} \frac{1}{d(v)\beta+1} \frac{\beta}{d(v)\beta+1} \notag \\
& = \frac{1}{2\beta}\left( 1- \frac{d(v)\beta-1}{d(v)\beta+1}\right) p_{T,\beta}(v,w) \label{eq:difftransition1}.
\end{align}  
If $v\neq \rho$ and $w$ is the parent vertex of $v$,
\begin{align}
\frac{\mathrm d}{\mathrm  d\beta} p_{T,\beta}(v,w) & = \frac{\mathrm d}{\mathrm d\beta} \frac{1}{d(v)\beta +1} = \frac{-d(v)}{(d(v)\beta+1)^2}
= \frac{1}{\beta} \frac{-d(v)\beta}{d(v)\beta+1} \frac{1}{d(v)\beta+1} \notag \\
& = \frac{1}{2\beta}\left( -1- \frac{d(v)\beta-1}{d(v)\beta+1}\right) p_{T,\beta}(v,w) \label{eq:difftransition2} .
\end{align}  
Finally, if $v=\rho$, then $w$ is the first descendant of $\rho$ and 
\begin{align}
\frac{\mathrm d}{\mathrm  d\beta} p_{T,\beta}(v,w)  = 0 = \frac{1}{2\beta}\left( 1- 1\right) p_{T,\beta}(v,w) \label{eq:difftransition3} .
\end{align}  
In any of the cases \eqref{eq:difftransition1}, \eqref{eq:difftransition2}, \eqref{eq:difftransition3}, we see that
\begin{align}
\frac{\mathrm d}{\mathrm  d\beta} p_{T,\beta}(v,w)  = \frac{1}{2\beta}\left( |w|-|v|- \Delta_\beta(v)\right) p_{T,\beta}(v,w) \label{eq:difftransitionfinal} ,
\end{align}  
with $\Delta_\beta(v)$ the local bias at vertex $v$. 
Returning to \eqref{eq:diffinner}, this implies
\begin{align}
\frac{\mathrm d}{\mathrm  d\beta} E_{T,\beta}[|X_\sigma|\mathbbm{1}_{\{\sigma=m\}}] 
& = \sum_z |z_m| \frac{\mathrm d}{\mathrm  d\beta} \prod_{i=0}^{m-1} p_{T,\beta}(z_i,z_{i+1}) \notag  \\
& = \sum_z |z_m| \frac{1}{2\beta}\left( \sum_{i=0}^{m-1} |z_{i+1}|-|z_i| - \Delta_{\beta}(z_i) \right) \prod_{i=0}^{m-1} p_{T,\beta}(z_i,z_{i+1}) \notag \\
& = \sum_z |z_m| \frac{1}{2\beta}\left(  |z_{m}| - \sum_{i=0}^{m-1}\Delta_{\beta}(z_i) \right) \prod_{i=0}^{m-1} p_{T,\beta}(z_i,z_{i+1}) \label{eq:diffinner2}. 
\end{align}
Note that the sum is over finitely many paths so that we may differentiate inside the sum. The final sum \eqref{eq:diffinner2} can be written as an expectation again and we obtain
\begin{align}
\frac{\mathrm d}{\mathrm  d\beta} E_{T,\beta}[|X_\sigma|\mathbbm{1}_{\{\sigma=m\}}] 
= \frac{1}{2\beta} E_{T,\beta}\left[ |X_\sigma|\left(  |X_\sigma|-  \sum_{i=0}^{\sigma-1}\Delta_{\beta}(X_i) \right) \mathbbm{1}_{\{\sigma=m\}}\right] .
\end{align}
The last expectation can by bounded by $\frac{1}{2\beta} E_{T,\beta}[\sigma^2 \mathbbm{1}_{\{\sigma=m\}}]$, which is summable over $m$ locally uniformly in $\beta>0$, since by Lemma \ref{lem:moment-uniform2nd}, 
\begin{align*}
\sup_{\beta\geq 0} E_{T,\beta}[\sigma^2]<C_T
\end{align*}
with a constant $C_T$ finite for almost $\mathrm{P}^*$-almost every tree $T$. This shows that 
\begin{align}
\frac{\mathrm d}{\mathrm  d\beta} \sum_{m=1}^\infty E_{T,\beta}[|X_\sigma|\mathbbm{1}_{\{\sigma=m\}}] 
= \sum_{m=1}^\infty \frac{\mathrm d}{\mathrm  d\beta}  E_{T,\beta}[|X_\sigma|\mathbbm{1}_{\{\sigma=m\}}] ,
\end{align}
and therefore 
\begin{align*}
\frac{\mathrm d}{\mathrm  d\beta}  E_{T,\beta}[|X_\sigma|] = \frac{1}{2\beta} E_{T,\beta}\left[ |X_\sigma|\left(  |X_\sigma|-  \sum_{i=0}^{\sigma-1}\Delta_{\beta}(X_i) \right)\right] .
\end{align*}
In the following, let us write $Y_\sigma =\frac{1}{2\beta}( |X_\sigma|-  \sum_{i=0}^{\sigma-1}\Delta_{\beta}(X_i))$. The expectation $E_{T,\beta}[|X_\sigma|Y_\sigma]$ is bounded by $2 E_{T,\beta}[\sigma^2]$ and using again the uniform moment bound of Lemma \ref{lem:moment-uniform2nd}, we may interchange differentiating with respect to $\beta$ and integrating with respect to the environment, so that
\begin{align*}
\frac{\mathrm d}{\mathrm  d\beta}  \bbE_{\beta}^*[|X_\sigma|]
= \frac{\mathrm d}{\mathrm  d\beta} \mathrm{E}^* \left[  E_{T,\beta}[|X_\sigma|]\right]
=  \mathrm{E}^* \left[ \frac{\mathrm d}{\mathrm  d\beta} E_{T,\beta}[|X_\sigma|]\right]
= \bbE_\beta^* \left[ |X_\sigma|\,Y_\sigma \right] 
\end{align*}
By analogous arguments,
\begin{align*}
\frac{\mathrm d}{\mathrm  d\beta}  \bbE_{\beta}^*[ \sigma ]
= \bbE_\beta^* \left[ \sigma\, Y_\sigma\right] .
\end{align*}
For the derivative of the speed this implies
\begin{align*}
\frac{\mathrm d}{\mathrm  d\beta} v_\infty(\beta)&  = \frac{\mathrm d}{\mathrm  d\beta} \frac{\bbE_{\beta}^*[ |X_\sigma| ]}{\bbE^*_{\beta}[ \sigma ]} 
= \frac{\bbE^*_{\beta}[ |X_\sigma|\,Y_\sigma ]\bbE^*_{\beta}[ \sigma ]- \bbE^*_{\beta}[ |X_\sigma| ]\bbE^*_{\beta}[ \sigma\, Y_\sigma ]}{\bbE^*_{\beta}[ \sigma ]^2} \\
& = \frac{\bbE^*_{\beta}[ |X_\sigma|\,Y_\sigma ]- v_\infty(\beta) \bbE^*_{\beta}[ \sigma\, Y_\sigma ]}{\bbE^*_{\beta}[ \sigma ]} \\
& = \frac{1}{\bbE^*_\beta[\sigma]}\bbE^*_\beta \left[ \big( |X_\sigma| -v_\infty(\beta) \sigma \big) Y_\sigma \right] ,
\end{align*}
which is the claimed derivative for $\beta>0$. 

For the derivative in $\beta=0$, we compute $\bbE^*_\beta[|X_\sigma|]$ and $\bbE^*_\beta[\sigma]$ up to order $\beta$. Under $\bbP^*_\beta$, the random walk starts at a marked root and the first step is always to the first descendant of the root. Recall that $M$ is the set of all marked vertices. We claim that only four types of path events need to be considered:
\begin{align*}
A_1 & = \{ X_1 \in M\}  \\
A_2 & = \{ X_1 \notin M, X_2=\rho\}  \\
A_3 & = \{ X_1 \notin M, X_2\in M, X_2\neq \rho \}  \\
A_4 & = \{ X_1, X_2\notin M, X_3=X_1,X_4=\rho\}  
\end{align*} 
Indeed, suppose that none of the events $A_i$ above occur. Then $X_2$ is a descendant of $X_1$, but not a marked vertex, and then either $X_3$ is a descendant of $X_2$ or $X_3=X_1$ and $X_4$ is a descendant of $X_3$. That is, 
\begin{align*}
(A_1\cup A_2\cup A_3 \cup A_4)^c \subset B_1\cup B_2 ,
\end{align*} 
with
\begin{align*}
B_1 & = \{X_1,X_2\notin M, |X_3|=|X_2|+1\} \\
B_2 & = \{X_1,X_2\notin M, X_3=X_1, |X_4|=|X_3|+1 \} .
\end{align*}
We have, denoting by $d_1,d_2$ the degrees of two distinct, non-marked vertices,  
\begin{align*}
\bbP^*_\beta(B_1) = \bbE^*_\beta\left[ (1-\alpha)^2 \frac{d_1\beta}{d_1\beta+1} \frac{d_2\beta}{d_2\beta+1} \right] \leq \beta^2 (1-\alpha)^2 \bbE^*_\beta[d_1]^2 
\end{align*}
and
\begin{align*}
\bbP^*_\beta(B_2) = \bbE^*_\beta\left[ (1-\alpha)^2 \left(\frac{d_1\beta}{d_1\beta+1}\right)^2 \frac{1}{d_2\beta+1} \right] \leq \beta^2 (1-\alpha)^2 \bbE_\beta[d_1^2] .
\end{align*}
Note that $\bbE^*_\beta[d_1]$ and $\bbE^*_\beta[d_1^2]$ are just the first and second moment, respectively, of $\nu$, which are finite by assumption. H\"older's inequality implies 
\begin{align} \label{eq:diffzerorest}
\bbE^*_\beta[|X_\sigma| \mathbbm{1}_{B_1\cup B_2}] \leq \bbE^*_\beta[\sigma \mathbbm{1}_{B_1\cup B_2}]
\leq \bbE^*_\beta(B_1\cup B_2)^{2/3} \bbE^*_\beta [\sigma^3]^{1/3} .
\end{align}
The moment bound of Lemma \ref{lem:moment-uniform2nd} shows that this upper bound is of order $\beta^{4/3}$. For the derivative at zero, it therefore suffices to consider the expectation of $|X_\sigma|$ and of $\sigma$ on one of the events $A_1,\dots ,A_4$. For the position, only on $A_1$ and $A_3$ holds $|X_\sigma |\neq 0$, so that
\begin{align*}
\bbE^*_\beta[|X_\sigma|\mathbbm{1}_{A_1\cup A_2 \cup A_3 \cup A_4}] 
& = \bbE^*_\beta[ |X_\sigma|\mathbbm{1}_{A_1}]
 + \bbE^*_\beta[ |X_\sigma|\mathbbm{1}_{A_3}] \\
& = \alpha + 2(1-\alpha)\alpha \bbE^*_\beta\left[ \frac{d_1\beta}{d_1\beta +1} \right] \\
& = \alpha + 2(1-\alpha)\alpha  \bbE^*_\beta[d_1] \beta + o(\beta) .
\end{align*}
Together with \eqref{eq:diffzerorest}, we obtain
\begin{align}
\frac{\mathrm d}{\mathrm d \beta}\bbE^*_\beta[|X_\sigma|] \big|_{\beta=0} = 2(1-\alpha)\alpha m(\nu) . 
\end{align}
For the expectation of $\sigma$, we have 
\begin{align*}
& \bbE^*_\beta[\sigma \mathbbm{1}_{A_1\cup A_2 \cup A_3 \cup A_4}] \\
& = \alpha + 2(1-\alpha) \bbE^*_\beta\left[ \frac{1}{d_1\beta +1} \right] + 2(1-\alpha)\alpha \bbE^*_\beta\left[ \frac{d_1\beta}{d_1\beta +1} \right] \\
& \qquad + 4 (1-\alpha)^2 \bbE^*_\beta\left[ \frac{d_1\beta}{d_1\beta +1} \frac{1}{d_2\beta +1} \frac{1}{d_1\beta +1}\right] \\
& = 2-\alpha - 2(1-\alpha) \bbE^*_\beta[d_1]\beta  + 2(1-\alpha)\alpha \bbE^*_\beta[d_1]\beta+ 4 (1-\alpha)^2 \bbE^*_\beta[d_1]\beta + o(\beta) ,
\end{align*}
which yields 
\begin{align}
\frac{\mathrm d}{\mathrm d \beta}\bbE^*_\beta[\sigma ] \big|_{\beta=0} = 2(1-\alpha)^2 m(\nu) . 
\end{align}
Finally, the derivative of the speed is given by 
\begin{align}
\frac{\mathrm d}{\mathrm d \beta}v_\infty(\beta) \big|_{\beta=0}
& =
\frac{\mathrm d}{\mathrm d \beta}\frac{\bbE^*_\beta[|X_\sigma |]}{\bbE^*_\beta[\sigma]} \big|_{\beta=0}
= \frac{[2(1-\alpha)\alpha m(\nu)] (2-\alpha) - \alpha [2(1-\alpha)^2m(\nu)]}{(2-\alpha)^2} 
\\
& = \frac{2\alpha (1-\alpha)m(\nu)}{(2-\alpha)^2} ,
\end{align}
which finishes the proof of Theorem \ref{thm:infinitespeeddiff}.
\qed

\section{Convergence of the speed}
\label{sec:convergence}

In this section we prove Theorem \ref{thm:speedconvergence}. 
Let $\beta>0$ and suppose $N$ is sufficiently large, such that $\beta>1/m_N$ and $v_N(\beta)$ is positive. We make use of level regeneration times as studied by \cite{Dembo2002largeDeviations}. They are random times $0<\tau_1<\tau_2<\dots $ characterized by  
\begin{align} \label{eq:defclassicalregtime}
|X_m| < |X_{\tau_k}| \text{ for all }m<\tau_k,\qquad
|X_m| \geq |X_{\tau_k}| \text{ for all } m\geq \tau_k, 
\end{align}
that is, at time $\tau_k$ the random walk reaches a generation it has not seen before and then never visits its parent vertex. In the transient regime, these regeneration times are finite a.s. and
\begin{align*}
\big( |X_{\tau_{k+1}}| - |X_{\tau_k}|, \tau_{k+1}-\tau_k\big)_{k\geq 1}
\end{align*} 
is an i.i.d. sequence under $\mathbb P_\beta^{(N)}$. As for the times $\sigma_k$ defined in Section \ref{sec:infinitespeed}, the increments until time $\tau_1$ have a different distribution. Let us denote by $\bar{\bbP}_\beta^{(N)}$ the annealed law, under which the underlying Galton-Watson tree is an  augmented tree $T^*$, that is, we attach an additional edge with conductance $\beta^{-1}$ to the root $\rho$, connecting it to an artificial parent vertex $\rho^*$, and the random walk is conditioned on never hitting $\rho^*$, that is, on the event 
\begin{align*}
R^c = \{X_n\neq \rho^* \text{ for all } n\geq 0\} . 
\end{align*} 
Then the distribution of $( |X_{\tau_{2}}| - |X_{\tau_1}|, \tau_{2}-\tau_1)$ under $\bbP_\beta^{(N)}$ is equal to the distribution of $(|X_{\tau_1}|,\tau_1)$ under $\bar{\bbP}_\beta^{(N)}$, and the speed of the random walk is given by 
\begin{align}
v_N(\beta) = \frac{\bbE_\beta^{(N)}[|X_{\tau_2}|-|X_{\tau_1}|]}{\bbE_\beta^{(N)}[\tau_2-\tau_1]}
= \frac{\bar{\bbE}_\beta^{(N)}[|X_{\tau_1}| ]}{\bar{\bbE}_\beta^{(N)}[\tau_1 ]} .
\end{align}
The same construction holds under $\bbP_\beta$, such that 
\begin{align}
v_\infty(\beta) = \frac{\bbE_\beta[|X_{\tau_2}|-|X_{\tau_1}|]}{\bbE_\beta[\tau_2-\tau_1]}
= \frac{\bar{\bbE}_\beta[|X_{\tau_1}|]}{\bar{\bbE}_\beta[\tau_1 ]} ,
\end{align}
with the analogous definition of $\bar{\bbP}_\beta$. We make use of these formulas to prove convergence of the speed. 
The convergence of the expectations (here and also later in Section \ref{sec:diffconvergence}) will follow by the following argument: If $Z_N, N\geq 1$ are random variables, such that the distribution of $Z_N$ under $\bar{\bbP}_\beta^{(N)}$ converges weakly to the distribution of a random variable $Z_0$ under $\bar{\bbP}_\beta$, and if 
\begin{align*}
\sup_{N\geq N_0} \bar{\bbE}_\beta^{(N)}[Z_N^r] <\infty
\end{align*}
for some $N_0\in \mathbb N$ and $r>1$, then
\begin{align}\label{eq:expectationconv}
\bar{\bbE}_\beta^{(N)}[Z_N] \xrightarrow[N\to \infty]{} \bar{\bbE}_\beta[Z_0] .
\end{align}
This follows readily by constructing a coupling of all random variables $Z_0,Z_1,\dots $ on a common probability space, such that the uniform bound on the $r$-th moment implies uniform integrability. 

The following lemma shows a convergence in distribution for the regeneration epochs.

\begin{lem}\label{lem:regweakconv}
Let $\beta>0$. For any $m\geq 1$, the distribution of 
\begin{align} 
\big( |X_1|,\dots ,|X_m|,\Delta_\beta(X_0),\dots ,\Delta_\beta (X_{m-1}), \mathbbm{1}_{\{\tau_1 =m \}} \big) 
\end{align}
under $\bar{\bbP}_\beta^{(N)}$ converges as $N\to \infty$ weakly to its distribution under $\bar{\bbP}_\beta$.
\end{lem}

\begin{proof}
Assume that $N$ is sufficiently large, so that $v_N(\beta)>0$. 
We need a particular labeling of marked vertices and set $\hat{d}(v) = -1$ if $v$ is marked, and $\hat d(v)=d(v)$ otherwise. 
For  $m\geq 1$, $z\in \mathbb{N}_0^m$ and $ d\in (\mathbb{N}_0\cup \{-1\})^{m}$, we consider events of the form 
\begin{align} \label{eq:pathevent}
A_{m,z,d}= \{|X_1|=z_1,\dots ,|X_m|=z _m, \hat d(X_0)=d_0,\dots ,\hat d(X_{m-1})=d_{m-1} \} .
\end{align}
We denote by 
\begin{align} \label{eq:defhittingtime}
\eta_k = \inf\{n\geq 0: |X_n|=k\}
\end{align}
the hitting time of generation $k$. 
Additionally, let $C_m$ be the event that until time $m$, $\rho^*$ has not been visited and the generation of $X_m$ has not been visited before time $m$, and for every $1\leq k<|X_m|$ there exists some $\ell \in (\eta_k,m)$ with $|X_\ell |= k-1$. The event $\{\tau_1=m\}$ implies that $C_m$ holds, so that 
\begin{align} \label{eq:pathevent2}
\bbP_\beta^{(N)}(A_{m,z,d} \cap \{\tau_1=m\} \mid R^c )
= \bbP_\beta^{(N)}(A_{m,z,d} \cap C_m \cap \{\tau_1=m\} \mid R^c ) 
\end{align}
with the analogous equality under $\bbP_\beta$. 
On $C_m$, no time smaller than $m$ can be a regeneration time, so that $\tau_1=m$ if and only if after hitting $X_m$, the random walk never hits $X_m^*$ again. The Markov property implies
\begin{align*}
\bbP_\beta^{(N)}(A_{m,z,d} \cap C_m \cap \{\tau_1=m\} \cap R^c ) &= \bbP_\beta^{(N)}(A_{m,z,d}\cap C_m \cap \{\tau_1=m\}) \\
& = \bbP_\beta^{(N)}(A_{m,z,d}\cap C_m) \bbP_\beta^{(N)}(R^c) ,
\end{align*} 
which yields
\begin{align*}
\bbP_\beta^{(N)}(A_{m,z,d} \cap C_m \cap \{\tau_1=m\} \mid R^c ) & = \bbP_\beta^{(N)}(A_{m,z,d}\cap C_m) .
\end{align*}
Although under $\bbP_\beta$ the process is not Markovian anymore, the label $\hat d(X_k)$ of the vertices determine the transition probabilities, and the analogous factorization of those gives the analogous identity
\begin{align*}
\bbP_\beta(A_{m,z,d} \cap C_m \cap \{\tau_1=m\} \mid R^c ) & = \bbP_\beta(A_{m,z,d}\cap C_m) .
\end{align*}
The labels in $d$ determine when marked vertices are visited. As $N\to \infty$, the transition probabilities at unmarked vertices do not change. At a marked vertex, the probability that the distance to the root increases by one converges to 1. This implies
\begin{align*}
\bbP_\beta^{(N)}(A_{m,z,d}\cap C_m) \xrightarrow[N\to \infty]{} \bbP_\beta(A_{m,z,d}\cap C_m) .
\end{align*}
Returning to \eqref{eq:pathevent2}, we obtain the convergence
\begin{align*}
\bbP_\beta^{(N)}(A_{m,z,d} \cap \{\tau_1=m\} \mid R^c )
\xrightarrow[N\to \infty]{} 
\bbP_\beta(A_{m,z,d} \cap \{\tau_1=m\} \mid R^c ), 
\end{align*}
and thus the weak convergence of the distribution of  
\begin{align*} 
\big( |X_1|,\dots ,|X_m|,\hat d(X_0),\dots ,\hat d (X_{m-1}), \mathbbm{1}_{\{\tau_1 =m \}} \big) 
\end{align*}
under $\bbP_\beta^{(N)} (\cdot \mid R^c)$ to the distribution under $\bbP_\beta(\cdot \mid R^c)$. The statement of the lemma follows noting that $\Delta_\beta(v)$ is a deterministic function of $\hat d(v)$ converging as $N\to \infty$ when $v$ is marked. 
\end{proof}

We complement the weak convergence in Lemma \ref{lem:regweakconv} with the uniform moment bound 
\begin{align} \label{eq:2ndmoment}
\sup_{N\geq N_0} \bar{\bbE}_\beta^{(N)}[\tau_1^2 ] < \infty
\end{align}
for the regeneration times provided by Lemma \ref{lem:momentuniform}. 
The weak convergence in Lemma \ref{lem:regweakconv} and the arguments for \eqref{eq:expectationconv} imply that
\begin{align*}
\bar{\bbE}_\beta^{(N)}[\tau_1 ] \xrightarrow[N\to \infty]{}  \bar{\bbE}_\beta[\tau_1 ] .
\end{align*}
Since $|X_{\tau_1}|\leq \tau_1$, the analogous argument can be applied to the regeneration position, showing
\begin{align*}
\bar{\bbE}_\beta^{(N)}[|X_{\tau_1}|] \xrightarrow[N\to \infty]{}  \bar{\bbE}_\beta[|X_{\tau_1}| ] .
\end{align*}
We therefore obtain 
\begin{align} \label{eq:regtimeconv}
v_N(\beta) = \frac{\bar{\bbE}_\beta^{(N)}[|X_{\tau_1}| ]}{\bar{\bbE}_\beta^{(N)}[\tau_1 ]}
 \xrightarrow[N\to \infty]{}  
\frac{\bar{\bbE}_\beta[|X_{\tau_1}| ]}{\bar{\bbE}_\beta[\tau_1  ]}
= v_\infty(\beta) .
\end{align}

\section{The derivative of the speed}
\label{sec:diffconvergence}

The existence of the derivative $v_\infty'$ is a consequence of Theorem \ref{thm:infinitespeeddiff}. Now assume that there exists a $\gamma>1$ with $\sum_{k=1}^\infty p_k\gamma^k<\infty$. Under this assumption, Bowditch and Tokushige \cite{bowditch2020differentiability} proved differentiability of $v_N$. 
We first recall a formula for the derivative of the speed given in Proposition 3.4 in \cite{bowditch2020differentiability}, noting that therein, the bias is parametrized as $\lambda=1/\beta$: 
\begin{align} \label{eq:bowditchtokushigediff}
v_N'(\beta) & = \frac{1}{2\beta \bbE_\beta^{(N)}[\tau_2-\tau_1]} \bbE_\beta^{(N)}\left[ \left( |X_{\tau_2}| - |X_{\tau_1}| - (\tau_2-\tau_1) v_N(\beta)\right)\left( |X_{\tau_2}| - |X_{\tau_1}| - \sum_{k=\tau_1}^{\tau_2-1} \Delta_\beta(X_k) \right)\right] \notag \\
& = \frac{1}{2\beta \bar{\bbE}_\beta^{(N)}[\tau_1]} \bar{\bbE}_\beta^{(N)}\left[ \left( |X_{\tau_1}|  - \tau_1 v_N(\beta)\right)\left( |X_{\tau_1}| - \sum_{k=0}^{\tau_1-1} \Delta_\beta(X_k) \right)\right] .
\end{align}
As in Section \ref{sec:convergence}, $\Delta_\beta(v)$ is the local bias at vertex $v$ and we write $\bar{\bbP}_\beta^{(N)}$ for the conditioned law $\bbP_\beta^{(N)}(\cdot \mid R^c)$. We already established in Section \ref{sec:convergence} the convergence
\begin{align*}
\bar{\bbE}_\beta^{(N)}[\tau_1 ] \xrightarrow[N\to \infty]{} \bar{\bbE}_\beta[\tau_1 ] .
\end{align*}
and $v_N(\beta)\to v_\infty(\beta)$. By Lemma \ref{lem:regweakconv}, the distribution of 
\begin{align*}
\left( |X_{\tau_1}|  - \tau_1 v_N(\beta)\right)\left( |X_{\tau_1}| - \sum_{k=0}^{\tau_1-1} \Delta_\beta(X_k) \right)
\end{align*}
under $\bar{\bbP}_\beta^{(N)}$ converges weakly to the distribution of 
\begin{align*}
\left( |X_{\tau_1}|  - \tau_1 v_\infty(\beta)\right)\left( |X_{\tau_1}| - \sum_{k=0}^{\tau_1-1} \Delta_\beta(X_k) \right)
\end{align*}
under $\bar{\bbP}_\beta$. We can bound
\begin{align*}
\left| \left( |X_{\tau_1}|  - \tau_1 v_N(\beta)\right)\left( |X_{\tau_1}| - \sum_{k=0}^{\tau_1-1} \Delta_\beta(X_k) \right)\right| \leq 4\tau_1^2 .
\end{align*}
Lemma \ref{lem:momentuniform} provides a uniform moment bound for $\tau_1^3$. We can then conclude as in \eqref{eq:expectationconv}, so that
\begin{align}
v_N'(\beta) \xrightarrow[N\to \infty]{} \frac{1}{2\beta \bar{\bbE}_\beta[\tau_1]} \bar{\bbE}_\beta\left[  \left( |X_{\tau_1}|  - \tau_1 v_\infty(\beta)\right)\left( |X_{\tau_1}| - \sum_{k=0}^{\tau_1-1} \Delta_\beta(X_k) \right)\right] .
\end{align}
We show that the limit agrees with the expression for the derivative $v_\infty'(\beta)$ in Theorem \ref{thm:infinitespeeddiff}. Define the two-dimensional process
\begin{align}
(W_{1,n},W_{2,n}) = \left( |X_n|-nv_\infty(\beta) , |X_{n}| - \sum_{k=0}^{n-1} \Delta_\beta(X_k) \right) .
\end{align}
The regeneration times $\tau_k$ decompose the trajectory of  $(W_{1,n},W_{2,n})_n$ into independent increments, 
\begin{align*}
(W_{1,\tau_k},W_{2,\tau_k}) - (W_{1,\tau_{k-1}},W_{2,\tau_{k-1}}), \qquad k\geq 1 . 
\end{align*}
Together with the higher moment bounds for $\tau_k-\tau_{k-1}$ proven in Lemma \ref{lem:momentuniform} one can show the  convergence  
\begin{align*}
\frac{1}{\sqrt{n}}(W_{1,n},W_{2,n}) \xrightarrow[n\to \infty]{d} (W_1,W_2)
\end{align*}
in distribution under $\bbP_\beta$, see \cite[p. 621]{bowditch2020differentiability} for the analogous argumentation under $\bbP_\beta^{(N)}$. The limit is a two-dimensional Gaussian random variable and the covariance between the different components is given by 
\begin{align*}
\bbE_\beta[W_1W_2] & =
\frac{\bbE_\beta[(W_{1,\tau_2}-W_{1,\tau_1})(W_{2,\tau_2}-W_{2,\tau_1})]}{\bbE_\beta[\tau_2-\tau_1]} 
= \frac{\bar{\bbE}_\beta[W_{1,\tau_1}W_{2,\tau_1}]}{\bar{\bbE}_\beta[\tau_1]} \\
&=  \frac{1}{\bar{\bbE}_\beta[\tau_1]} \bar{\bbE}_\beta\left[  \left( |X_{\tau_1}|  - \tau_1 v_\infty(\beta)\right)\left( |X_{\tau_1}| - \sum_{k=0}^{\tau_1-1} \Delta_\beta(X_k) \right)\right] .
\end{align*}
Instead of the vertex regeneration times, we may consider the subsequent hitting times $\sigma_k$ of marked vertices as in Section \ref{sec:infinitespeed}. They also decompose the trajectory into independent segments, and $\sigma_k-\sigma_{k-1}$ has bounded moments of any order by Lemma \ref{lem:moment-uniform2nd}. The analogous argumentation then yields
\begin{align*}
\bbE_\beta[W_1W_2] =\frac{\bbE_\beta[(W_{1,\sigma_2}-W_{1,\sigma_1})(W_{2,\sigma_2}-W_{2,\sigma_1})]}{\bbE_\beta[\sigma_2-\sigma_1] }
= \frac{{\bbE}^*_\beta[W_{1,\sigma_1}W_{2,\sigma_1}]}{{\bbE}^*_\beta[\sigma_1]} \\
\frac{1}{\bbE_\beta^*[\sigma_1]} \bbE_\beta^*\left[ \left( |X_{\sigma_1}| - \sigma_1 v_\infty(\beta)\right)\left( |X_{\sigma_1}| - \sum_{k=0}^{\sigma_1-1} \Delta_\beta(X_k) \right)\right] ,
\end{align*}
which, after dividing by $2\beta$, gives the formula for the derivative as stated in Theorem \ref{thm:infinitespeeddiff}. 
This concludes the proof of $v_N'(\beta)\to v_\infty'(\beta)$.

\section{Monotonicity of the speed}
\label{sec:monotonicity}

In this section, we assume that $\nu=\delta_d$ for some $d\geq 1$. 
We begin with a slight reformulation of the speed formula. Starting from a marked root $\rho$, the first visited vertex is then marked with probability $\alpha$ or has $d$ descendants with probability $(1-\alpha)$. In the former case, $\sigma=|X_\sigma|=1$, with $\sigma$ as defined in \eqref{eq:hittingsigma}. Using Formula \eqref{eq:speedformula}, we have
\begin{align} \label{eq:couplingspeed}
v_\infty(\beta) = \frac{\bbE^*_{\beta}[|X_\sigma|]}{\bbE^*_{\beta}[\sigma]}
= \frac{1+(1-\alpha)\hat{\bbE}_{\beta}[|X_\sigma|]}{1+(1-\alpha)\hat{\bbE}_{\beta}[\sigma]} ,
\end{align}
where the expectation is with respect to a measure $\hat{\bbP}_\beta^\infty$, under which the random walk starts at the root $r$ of an augmented Galton-Watson tree $T_r$, with $r$ unmarked and attached to $r$ an additional marked parent vertex $r^*$. Lemma \ref{lem:criticalvalue} implies that the expectation $\bbE_\beta^*[\sigma]$, and then also the expectation of $\sigma$ under $\hat{\bbP}_\beta$, is finite. 
  
The connected component $T_r^d$ of $r$ in $T_r$ consisting of unmarked vertices is a Galton-Watson tree with binomial offspring distribution with parameters $d,1-\alpha$. 
Until $(X_n)_n$ hits a marked vertex, it only visits vertices with $d$ descendants, where it has local bias
\begin{align*}
\Delta(\beta) = \frac{d\beta-1}{d\beta+1} 
\end{align*} 
(slightly changing the notation for the bias). 
This implies that the process $(|X_{n\wedge \sigma}|-(n\wedge \sigma)\Delta(\beta))_n$ is a martingale, so that by optional stopping,
\begin{align*}
\hat{\bbE}_{\beta}[|X_\sigma|] = \hat{\bbE}_{\beta}[\sigma] \cdot \Delta(\beta) .
\end{align*}
This allows to rewrite \eqref{eq:couplingspeed} as
\begin{align} \label{eq:couplingspeed2}
v_\infty(\beta) = 
 \frac{1+(1-\alpha) \Delta(\beta) \hat{\bbE}_{\beta}[\sigma]}{1+(1-\alpha)\hat{\bbE}_{\beta}[\sigma]} .
\end{align}
We now compare the speed for two different bias parameters $\beta_1<\beta_2$. To simplify notation, we write $\hat p_i=d\beta_i/(d\beta_i+1)$ for the probability that the random walk with bias $\beta_i$ moves from a vertex with $d$ descendants to an offspring vertex, such that $\Delta(\beta_i)=2\hat p_i-1$, and we also write $\hat{\bbE}_{i}[\sigma]$ for $\hat{\bbE}_{\beta_i}[\sigma]$. The aim is to show positivity of the difference
\begin{align*}
& v_\infty(\beta_2) - v_\infty(\beta_1) = 
 \frac{1+(1-\alpha) (2\hat p_2-1) \hat{\bbE}_{2}[\sigma]}{1+(1-\alpha)\hat{\bbE}_{2}[\sigma]} -  \frac{1+(1-\alpha) (2\hat p_1-1) \hat{\bbE}_{1}[\sigma]}{1+(1-\alpha)\hat{\bbE}_{1}[\sigma]} \\
 & = 2(1-\alpha) \frac{\hat{\bbE}_{1}[\sigma]-\hat{\bbE}_{2}[\sigma] + \hat p_2 \hat{\bbE}_{2}[\sigma]-\hat p_1 \hat{\bbE}_{1}[\sigma] + (1-\alpha)(\hat p_2-\hat p_1)\hat{\bbE}_{1}[\sigma]\hat{\bbE}_{2}[\sigma]}{(1+(1-\alpha)\hat{\bbE}_{1}[\sigma])(1+(1-\alpha)\hat{\bbE}_{2}[\sigma])}\\
 & = 2(1-\alpha) \frac{(\hat p_2-\hat p_1)\hat{\bbE}_{1}[\sigma](1+(1-\alpha)\hat{\bbE}_{2}[\sigma]) + (1-\hat p_2)(\hat{\bbE}_{1}^\infty[\sigma]-\hat{\bbE}_{2}[\sigma])}{(1+(1-\alpha)\hat{\bbE}_{1}[\sigma])(1+(1-\alpha)\hat{\bbE}_{2}[\sigma])} \\
 & \geq 2(1-\alpha) \frac{(\hat p_2-\hat p_1) - (1-\hat p_2)(\hat{\bbE}_{1}[\sigma]-\hat{\bbE}_{2}[\sigma])}{(1+(1-\alpha)\hat{\bbE}_{1}[\sigma])(1+(1-\alpha)\hat{\bbE}_{2}[\sigma])},
\end{align*} 
where for the last inequality we used that $\hat p_2>\hat p_1$ and $\hat{\bbE}_{i}[\sigma]\geq 1$. Monotonicity of the speed is then proven once we show that for $\alpha$ sufficiently close to 1
\begin{align} \label{eq:couplingbound1}
 \big| \hat{\bbE}_{1}[\sigma]-\hat{\bbE}_{2}[\sigma]\big| \leq \hat p_2-\hat p_1
\end{align}
for all $0\leq \beta_1<\beta_2<\infty$ (that is, for all $0\leq \hat p_1<\hat p_2<1$).

We introduce a coupling of two processes $(X_{1,n})_{n\geq 0}$ and $(X_{2,n})_{n\geq 0}$ on $T_r$, such that  $(X_{1,n})_{n\geq 0}$ has marginal law $P_{T_r,\beta_1}^r$ and $(X_{2,n})_{n\geq 0}$ has law $P_{T_r,\beta_2}^r$. For this, assume that $U_0,U_1,\dots $ are independent random variables, uniformly distributed on $(0,1)$ on some probability space with law denoted by $\hat P_{T_r}$. We may then assume that the two processes satisfy the following. 
\begin{itemize}
\item[(1)] When $X_{1,n}$ and $X_{2,n}$ are both at the same unmarked vertex $v$, and if $\hat p_1<U_n\leq \hat p_2$, $X_{1,n}$ moves to the parent vertex of $v$ and $X_{2,n}$ moves to a uniformly chosen descendant of $v$.
\item[(2)] When $X_{1,n}$ and $X_{2,n}$ are both at the same unmarked vertex $v$, case (1) has not occurred so far, and if $U_n \leq \hat p_1$, they both move to the same uniformly chosen descendant of $v$.
\item[(3)] When $X_{1,n}$ and $X_{2,n}$ are both at the same unmarked vertex $v$, case (1) has not occurred so far and if $U_n > \hat p_2$, they both move to the parent vertex of $v$.
\end{itemize}
The choice of the offspring vertices in cases (1) and (2) are independent of each other and independent of the random variables $U_n,n\geq 0$. Once case (1) occurs or the processes hit a marked vertex, they continue independently of each other with transition probabilities determined by $P_{T_r,\beta_1}^r$ and  $P_{T_r,\beta_2}^r$, respectively. It is straightforward to check that then the two processes have the correct marginal laws. 

Let $\sigma_i$ be the time when process $(X_{i,n})_n$ hits a marked vertex and let $D$ be the first time  when case (1) above occurs. Then $\sigma_1,\sigma_2$ and $D$ are stopping times with respect to the filtration generated by $(U_n,X_{1,n},X_{2,n})_n$. We have, conditioning on $D$,
\begin{align*}
 \hat{\bbE}_{1}[\sigma]-\hat{\bbE}_{2}[\sigma]  & = \mathrm{E} \left[ \hat E_{T_r}[\sigma_1 - \sigma_2]\right] = \mathrm{E} \left[ \hat E_{T_r}[\sigma_1 - \sigma_2]\mathbbm{1}_{\{|T_r^d|>1\}}\right] \\
 & = \mathrm{E} \left[ \hat E_{T_r}\big[ \hat E_{T_r}[\sigma_1 - \sigma_2|D]\mathbbm{1}_{\{D<\sigma_1\wedge \sigma_2\}}\big]\mathbbm{1}_{\{|T_r^d|>1\}}\right]
\end{align*}
If $D$ is smaller than $\sigma_1$ and $\sigma_2$, the difference between $\sigma_1$ and $\sigma_2$ is the time after one of the processes has hit a marked vertex, until the other one hits a marked vertex. 
The hitting times $\sigma_i$ are hitting times of the complement of $T_r^d$, and if $|T_r^d|=1$, we have $\sigma_1=\sigma_2$. 
We use the bound in Lemma \ref{lem:moment-finitesetescape}, which shows that for any starting vertex $v\in T_r^d$ and for any bias $\beta>0$,
\begin{align} \label{eq:monotonicityhitting}
E_{T,\beta}^v [\sigma] \leq 2|T_r^d|^2 .
\end{align} 
Together with the strong Markov property we obtain
\begin{align*}
\left| \hat{\bbE}_{1}[\sigma]-\hat{\bbE}_{2}[\sigma]\right| & \leq  \mathrm{E} \left[ \hat E_{T_r}\big[ \hat E_{T_r}[|\sigma_1 - \sigma_2|\, |D]\mathbbm{1}_{\{D<\sigma_1\wedge \sigma_2\}}\big]\mathbbm{1}_{\{|T_r^d|>1\}} \right] \\
& \leq \mathrm{E} \left[ \hat E_{T_r}\big[ 2|T_r^d|^2  \mathbbm{1}_{\{D<\sigma_1\wedge \sigma_2\}}\big] \mathbbm{1}_{\{|T_r^d|>1\}}\right] \\
& = \mathrm{E} \left[2|T_r^d|^2  \hat P_{T_r}(D<\sigma_1\wedge \sigma_2)\mathbbm{1}_{\{|T_r^d|>1\}} \right] 
\end{align*}
Until time $\sigma_1\wedge \sigma_2$, the probability that $D$ occurs at a given time is equal to the probability of case (1) in the construction of the coupling, which is $\hat p_2-\hat p_1$. This gives
\begin{align*}
\hat P_{T_r}(D<\sigma_1\wedge \sigma_2) \leq (\hat p_2-\hat p_1) \hat E_{T_r}[\sigma_1]\leq  2(\hat p_2-\hat p_1) |T_r^d|^2 ,
\end{align*}
where we used again the bound \eqref{eq:monotonicityhitting}. We obtain 
\begin{align}
\left| \hat{\bbE}_{1}[\sigma]-\hat{\bbE}_{2}[\sigma]\right| \leq 4(\hat p_2-\hat p_1)\mathrm{E} \left[|T_r^d|^4\mathbbm{1}_{\{|T_r^d|>1\}} \right] . 
\end{align}
In view of \eqref{eq:couplingbound1}, the proof of monotonicity is completed once we show that 
\begin{align} \label{eq:couplingbound2}
4 \mathrm{E} \left[ |T_r^d|^4\mathbbm{1}_{\{|T_r^d|>1\}} \right] <1
\end{align}
for $\alpha$ sufficiently close to 1. 

Recall that $T_r^d$ is the subtree of $T_r$ containing $r$ formed by unmarked vertices. In particular, $T_r^d$ is a Galton-Watson tree with binomial offspring distribution with parameters $d,1-\alpha$.  
For the total progeny of $T_r^d$, the hitting time formula \cite[Theorem 3.13]{vanDerHofstad2016random1} yields
\begin{align*}
P( |T_r^d| = m) = \frac{1}{m} P(W_1+\dots +W_m = m-1) ,
\end{align*}
where $W_i$ are independent, binomially distributed random variable with parameters $d,1-\alpha$. Since the sum is binomially distributed with parameters $md,1-\alpha$, 
\begin{align*}
P( |T_r^d| = m) = \frac{1}{m} {md \choose m-1} (1-\alpha)^{m-1} \alpha^{md-m+1} .
\end{align*}
Using ${n\choose k} \leq (\mathrm{e} n/k)^{k}$, we get
\begin{align*}
P( |T_r^d| = m) \leq  \frac{1}{m} \left(\frac{\mathrm{e} md}{m-1}\right)^{m-1} (1-\alpha)^{m-1} 
\leq \frac{1}{m} (2\mathrm{e} d (1-\alpha))^{m-1} .
\end{align*}
This implies
\begin{align*}
4 \mathrm{E} \left[|T_r^d|^4\mathbbm{1}_{\{|T_r^d|>1\}} \right] 
& \leq 4  \sum_{m=2}^\infty m^3 (2\mathrm{e} d (1-\alpha))^{m-1} \\
& \leq 8  \sum_{m=2}^\infty 4^{m-1}  (2\mathrm{e} d (1-\alpha))^{m-1} \\
& = 8 \frac{8\mathrm{e} d (1-\alpha)}{1-8\mathrm{e} d (1-\alpha)}
\end{align*}
for $8\mathrm{e} d (1-\alpha)<1$. This upper bound is less than 1 provided
\begin{align}
1-\alpha < \frac{1}{72 \mathrm{e} d} .
\end{align}
If this bound holds, we have \eqref{eq:couplingbound2}, which implies \eqref{eq:couplingbound1}, and then $v_\infty(\beta_1)<v_\infty(\beta_2)$.

\section{Moment bounds} 
\label{sec:moments}

\begin{lem}\label{lem:moment-finitesetescape}
Let $T$ be an augmented tree without leaves, $A$ a finite connected subset of the vertices of $T$ containing the root $\rho$, but not the artificial parent $\rho^*$, and such that every vertex in $A$ has finite degree in $T$. Let $\eta(A^c)$ be the hitting time of $A^c$. Then for any $v\in A$, we have for all $\beta \geq 0$  
\begin{align} 
E_{T,\beta}^v [\eta({A^c})] \leq 2|A|^2 .
\end{align}
\end{lem}

\begin{proof}
First, suppose $\beta \geq 1$. We use a quadratic submartingale comparable with the classical case in \cite[Section 17.3]{levin2017markov}. For any $w\in A$, we have 
\begin{align*}
E_{T,\beta}^w\big[ |X_1|^2-|w|^2] = 1 + 2|w| \frac{d(w)\beta -1}{d(w)\beta+1} \geq 1, 
\end{align*}
which implies that 
\begin{align*}
|X_{n\wedge \eta({A^c})}|^2 - |v|^2 - (n\wedge \eta({A^c})) , \qquad n\geq 0
\end{align*}
is a submartingale under $P_{T,\beta}^v$ with respect to the filtration generated by $(X_n)_{n\geq 0}$. Thus,
\begin{align*}
E_{T,\beta}^v[\eta({A^c})\wedge n] \leq E_{T,\beta}^v[|X_{n\wedge \eta({A^c})}|^2 ] . 
\end{align*}
Since $|X_{n\wedge \eta({A^c})}|\leq |A|$, we get by monotone and dominated convergence
\begin{align*}
E_{T,\beta}^v [\eta({A^c})] \leq E_{T,\beta}^v[|X_{\eta({A^c})}|^2 ]\leq |A|^2 .
\end{align*}
Now, suppose $\beta\leq 1$. Let $T_A$ denote the subtree of $T$ when all vertices of $A^c$ are removed and only the artifical parent $\rho^*$ is kept. Then,
\begin{align*}
E_{T,\beta}^v [\eta({A^c})] \leq E_{T_A,\beta}^v [\eta({A^c})] .
\end{align*}
In $T_A$, the random walk can only exit $A$ by hitting $\rho^*$. For any $w\in T_A$, the mean time to reach $w^*$ in $T_A$ can be computed as
\begin{align*}
E_{T_A,\beta}^w [\eta(w^*)] = 1+2 \sum_{k=1}^\infty \beta^k G_k(w)\leq  1+2 \sum_{k=1}^\infty  G_k(w) ,  
\end{align*}
where $G_k(w)$ is the number of descendants of $w$ in $T_A$ at distance $k$. An upper bound is
\begin{align*}
E_{T_A,\beta}^w [\eta(w^*)] \leq 1+2(|A|-1) \leq 2|A| .
\end{align*} 
Since in order to reach $\rho^*$ from $v$, at most $|A|$ parent vertices have to be reached, we obtain
\begin{align*}
E_{T_A,\beta}^v [\eta(A^c)] = E_{T_A,\beta}^v [\eta(\rho^*)] \leq 2|A|^2 ,
\end{align*}
which finishes the proof in the case $\beta\leq 1$. 
\end{proof}

\begin{lem}\label{lem:moment-uniform2nd}
If $p_0=0$, we have for any $q\geq 1$
\begin{align*} 
\mathrm{E}^*\left[ \sup_{\beta\geq 0} E_{T,\beta} [\sigma^q] \right]<\infty .
\end{align*}
\end{lem}

\begin{proof}
Since $\sigma=1$ if the first visited vertex is marked, it suffices to show the integrability of 
\begin{align*}
\sup_{\beta\geq 0} E_{T^*,\beta} [\sigma^q] , 
\end{align*} 
where, as introduced in Section \ref{sec:monotonicity}, $T^*$ is an augmented Galton-Watson tree with unmarked root $\rho$ and marked artificial parent $\rho^*$. Due to the independence of the offspring numbers in the Galton-Watson tree, we may encode the randomness of the environment $T^*$ in the following way: let $D=(D_1,D_2,\dots)$ be a sequence of i.i.d. random variables with marginal law as the offspring law under $\mathrm{P}$. We grow the tree as the random walk discovers new vertices, attaching descendants according to the sequence $D$. Let $\theta_n$ be the $n$-th time the random walk visits a vertex it has never seen before. When $X_{\theta_n}=\rho^*$, we stop the random walk and set $\theta_{k}=\theta_n$ for all $k>n$. When $X_{\theta_n}=v\neq \rho^*$, we attach to $v$ a number $D_n$ of descendants. Let $P_{D,\beta}$ be the law of this biased random walk on a growing tree, starting on an augmented, unmarked root with $D_0\sim \nu$ descendants. We have then
\begin{align*} 
\mathrm{E}\left[ \sup_{\beta\geq 0} E_{T^*,\beta} [\sigma^q] \right] = \mathrm{E}\left[ \sup_{\beta\geq 0} E_{D,\beta} [\sigma^q] \right] .
\end{align*}
Let $G$ be the smallest index $n$, such that $X_{\theta_n}$ is a marked vertex, that is,
\begin{align*}
G:= \inf\{n\geq 1  : X_{\theta_n} \in M\}
\leq \inf\{n\geq 1  : D_n=\infty\}=:\hat G .
\end{align*} 
Note that $X_{\theta_n}$ can be marked if the random walk has reached $\rho^*$. Therefore, the random variable $G$ depends on the trajectory of the random walk, whereas $\hat G$ is independent of it. Since each $D_n$ is infinite with probability $\alpha$, $\hat G$ is geometrically distributed with success parameter $\alpha$. The hitting time $\sigma$ can then be written as the sum of steps between the visits to new vertices as
\begin{align*}
\sigma = \sum_{k=1}^G (\theta_k-\theta_{k-1}) ,
\end{align*}
with $\theta_0=0$. The moment can be bounded by Jensen's inequality and using $G\leq \hat G$:
\begin{align*}
E_{D,\beta} [\sigma^q] & = E_{D,\beta} \left[ \left(\sum_{k=1}^G (\theta_k-\theta_{k-1})\right)^q\right] 
\leq E_{D,\beta} \left[ G^{q-1} \sum_{k=1}^G (\theta_k-\theta_{k-1})^q\right]  \\
& \leq E_{D,\beta} \left[ \hat G^{q-1} \sum_{k=1}^{\hat G} (\theta_k-\theta_{k-1})^q\right]  = \hat G^{q-1} \sum_{k=1}^{\hat G} E_{D,\beta}[(\theta_k-\theta_{k-1})^q]. 
\end{align*}
We next find a moment bound for the difference $(\theta_k-\theta_{k-1})$. At time $\theta_{k-1}$, the random walk has either reached $\rho^*$ (in which case $\theta_k=\theta_{k-1}$), or it has discovered $k$ vertices, so that Lemma \ref{lem:moment-finitesetescape} implies
\begin{align*}
E_{D,\beta}[\theta_k-\theta_{k-1}] \leq 2k^2 .
\end{align*}
Similar to \cite[Section 2.4.3]{AldFil02}, we obtain by the Markov inequality for any $m\geq 1$, after conditioning on the vertices  discovered at time $\theta_{k-1}$, 
\begin{align*}
P_{D,\beta}(\theta_k-\theta_{k-1}\geq m \cdot (4k^2) |\theta_k-\theta_{k-1}\geq (m-1) \cdot (4k^2) ) \leq \frac{1}{4k^2} E_{D,\beta}[\theta_k-\theta_{k-1}]  \leq \frac{1}{2} , 
\end{align*} 
so that by induction
\begin{align*}
P_{D,\beta}(\theta_k-\theta_{k-1}\geq  4 k^2m)\leq 2^{-m}.  
\end{align*}
This implies for the $q$-th moment
\begin{align} \label{eq:thetatailbound}
E_{D,\beta}[(\theta_k-\theta_{k-1})^q ] 
& \leq \sum_{m=1}^\infty m^q P_{D,\beta}( \theta_k-\theta_{k-1} \geq m) \notag \\
& \leq 4k^2  \sum_{m=1}^\infty (4k^2m)^q P_{D,\beta}( \theta_k-\theta_{k-1} \geq 4k^2m)\notag  \\
& \leq (4k^2)^{q+1}  \sum_{m=1}^\infty m^q 2^{-m}  \notag \\
& = C_qk^{2q+2} ,
\end{align}
for a finite constant $C_q$. Returning to the moment of $\sigma$, we bound
\begin{align*}
E_{D,\beta} [\sigma^q]  
\leq \hat G^{q-1} \sum_{k=1}^{\hat G} E_{D,\beta}[(\theta_k-\theta_{k-1})^q]
\leq \hat G^{q-1} \sum_{k=1}^{\hat G}C_qk^{2q+2}  \leq  C_q\hat G^{3q+2} .  
\end{align*}
Recall that $\hat G$ is geometrically distributed with parameter $\alpha$ and this upper bound is independent of the bias parameter $\beta$. This implies that 
\begin{align*}
\mathrm{E}\left[ \sup_{\beta\geq 0} E_{T^*,\beta} [\sigma^q] \right] = \mathrm{E}\left[ \sup_{\beta\geq 0} E_{D,\beta} [\sigma^q] \right] \leq C_q \mathrm{E} \big[ \hat G^{3q+2}\big]  
\end{align*}
is finite.
\end{proof}

\begin{lem}\label{lem:momentuniform-height}
Suppose $p_0=0$. For any $\beta>0$ there exists $N_0\in \mathbb N$ and $\theta>0$, such that   
\begin{align*}
\sup_{N\geq N_0} \bar{\mathbb{E}}^{(N)}_{\beta}\big[e^{\theta |X_{\tau_1}|}\big]<\infty . 
\end{align*}
\end{lem}

\begin{proof}
We start with a uniform bound for the probability of never hitting the artificial parent of the root in an augmented tree. For a tree $T$ with root $\rho$ we denote the escape probability in the augmented tree $T^*$ by
\begin{align} \label{eq:defescapeprob}
e_\beta(T) = P_{T^*,\beta}(R^c) = P_{T^*,\beta}(X_n\neq \rho^* \text{ for all } n\geq 0 ) = \frac{C_{T^*,\beta}(\rho,\infty)}{\beta^{-1}+C_{T^*,\beta}(\rho,\infty)} , 
\end{align}
with $C_{T^*,\beta}(\rho,\infty)$ the effective conductance between $\rho$ and infinity. Choose $N_0$ so large that the tree $T^*$ is $\mathrm{P}^{(N_0)}$-almost surely transient and this conductance is positive. 

There is a straightforward coupling of the underlying Galton-Watson trees for different $N$: given a Galton-Watson tree $T_N$ with offspring law $\nu_N$, attach to all marked vertices of $T_N$ an additional offspring and to that offspring a Galton-Watson tree with offspring law $\nu_{N+1}$. The resulting tree $T_{N+1}$ is a Galton-Watson tree with offspring law $\nu_{N+1}$, and $T_N$ is a subtree of $T_{N+1}$. Attaching to each marked vertex infinitely many offspring and to those descendants independent trees with law $\mathrm{P}$ yields a Galton-Watson tree $T_\infty$ with law $\mathrm{P}$, and $T_N\subset T_\infty$ for all $N$. Note that the above increasing coupling cannot be applied when $p_0>0$ and the trees are conditioned on survival. 

The same increasing coupling can be applied to augmented trees $T_N^*$. 
By Rayleigh's monotonicity law, $C_{T_N^*,\beta}(\rho,\infty)$ is increasing in $N$ and converges to $C_{T_\infty^*,\beta}(\rho,\infty)$. This shows that 
\begin{align*}
\bbP_\beta^{(N)}(R^c) = \mathrm{E}^{(N)} \left[ \frac{C_{T^*,\beta}(\rho,\infty)}{\beta^{-1}+C_{T^*,\beta}(\rho,\infty)} \right]  
\end{align*}
is increasing in $N$, so that 
\begin{align} \label{eq:returnprobbound}
\inf_{N\geq N_0} \bbP_\beta^{(N)}(R^c) = \bbP_\beta^{(N_0)}(R^c)=:r_0 >0 , 
\end{align}
and bounded convergence yields
\begin{align} \label{eq:returnprobconv}
\bbP_\beta^{(N)}(R^c)  \xrightarrow[N\to \infty]{}
 \bbP_\beta(R^c) .
\end{align} 
In the following, we frequently use the uniform bound \eqref{eq:returnprobbound} to simplify
\begin{align} \label{eq:conditionallawbound}
\bar\bbE_\beta^{(N)} [Z] = \bbE_\beta^{(N)} [Z\mathbbm{1}_{R^c}] \bbP_\beta^{(N)}(R^c)^{-1} \leq r_0^{-1} \bbE_\beta^{(N)} [Z]
\end{align}
for $N\geq N_0$ and $Z$ a nonnegative random variable. In particular, the statement of the lemma is proven once we show 
 \begin{align*}
\sup_{N\geq N_0} {\mathbb{E}}^{(N)}_{\beta}[e^{\theta |X_{\tau_1}|}]<\infty 
\end{align*}
for a $\theta>0$. We proceed analogous to \cite[Section 4.2]{glatzel2024CLT}: when the random walk reaches a generation it has never seen before in a vertex $v$, it has probability $\bbP_\beta^{(N)}(R^c)$ to regenerate. If the walker returns to $v^*$, the next potential generation for a regeneration point is  $|v| + H_v+1$, with $H_v$ the same distribution as 
\begin{align*}
H = \max_{0\leq n\leq \eta^*} |X_n|
\end{align*}
under $\bbP_\beta^{(N)}$ conditioned on $R=\{\eta^*<\infty\}$, with $\eta^*$ the hitting time of $\rho^*$. The independence of such excursions leads to 
\begin{align} \label{eq:regheight}
\bbE_\beta^{(N)}\big[ e^{\theta |X_{\tau_1}|}\big] = \bbE_\beta^{(N)}\left[ \exp\left\{ \theta +\theta \sum_{k=1}^{G-1} (H_k+1)\right\}\right] , 
\end{align}
with $G$ a geometric random variable with success parameter $\bbP_\beta^{(N)}(R^c)\geq r_0$ and $H_1,H_2,\dots $ i.i.d. with the same marginal distribution as $H$ under $\bbP_\beta^{(N)}(\cdot \mid R)$, independent of $G$. Since $G$ has exponential moments, it suffices to show exponential moments of $H$. Let $G_m(\rho)$ be the set of vertices of $T$ in generation $m$ and $\eta_m$ the hitting time of generation $m$ as defined in \eqref{eq:defhittingtime}, then by the Markov property,
\begin{align}\label{eq:regheightrecursion}
\bbP_\beta^{(N)} (H\geq m \mid R) & \leq \bbP_\beta^{(N)}(R)^{-1} \bbP_\beta^{(N)} (\eta_m<\eta^*<\infty)\notag  \\
& \leq \bbP_\beta^{(N)}(R)^{-1} \mathrm{E}^{(N)}\left[ \sum_{v\in G_m(\rho)} P^v_{T^*,\beta}(\eta^*<\infty)P_{T^*,\beta}(X_{\eta_m}=v)  \right] \notag \\
& \leq \bbP_\beta^{(N)}(R)^{-1} \mathrm{E}^{(N)}\left[ \sum_{v\in G_m(\rho)} P^v_{T^*,\beta}(\eta^*<\infty)^sP_{T^*,\beta}(X_{\eta_m}=v)  \right]^{1/s} ,  
\end{align}
where the last bound is Jensen's inequality for a $s\geq 2$ later to be specified. 
The probability $\bbP_\beta^{(N)}(R)$ is bounded from below by 
\begin{align*}
(1-\alpha) p_\ell \frac{1}{\beta \ell +1} >0, 
\end{align*}
with $\ell$ some integer such that $p_\ell>0$. Therefore, it suffices to bound  
\begin{align*}
W_m = \mathrm{E}^{(N)}\left[ \sum_{v\in G_m(\rho)} P^v_{T^*,\beta}(\eta^*<\infty)^s\right] . 
\end{align*}
We find a recursive bound for $ W_m$ inspired by arguments from \cite{Dembo2002largeDeviations}, which are based on ideas of \cite{grimmett1984random}. 
For a vertex $v=v_m$ in generation $m$, let $v_0,v_1,\dots ,v_m$ be the vertices on the path from $\rho=v_0$ to $v$. Then
\begin{align*} 
P^v_{T^*,\beta}(\eta^*<\infty) = \prod_{i=0}^m P^{v_i}_{T^*,\beta}(\eta_{i-1}<\infty) \leq \prod_{i=0}^{m-1} P^{v_i}_{T^*(v_{i+1}),\beta}(\eta_{i-1}<\infty) , 
\end{align*}
where $T^*(v_{i+1})$ is the subtree of $T^*$ with descendants of  $v_{i+1}$ removed, and $\eta_{-1}$ is the hitting time $\eta^*$ of $\rho^*$. This bound implies
\begin{align*}
W_m\leq \hat W_m := \mathrm{E}^{(N)}\left[  \sum_{v\in G_m(\rho)} \prod_{i=0}^{m-1} P^{v_i}_{T^*(v_{i+1}),\beta}(\eta_{i-1}<\infty)^s \right] .
\end{align*}
Removing descendants of vertex $v_{i+1}$ renders the hitting probabilities independent under $\mathrm{P}^{(N)}$. By conditioning on the different subtrees, we obtain
\begin{align*}
\hat W_m 
& = \mathrm{E}^{(N)}\left[   \sum_{v\in G_m(\rho)} \prod_{i=0}^{m-1} P^{v_i}_{T^*(v_{i+1}),\beta}(\eta_{i-1}<\infty)^s \right] \\
& = 
\mathrm{E}^{(N)}\left[  \sum_{v_{m-1}\in G_{m-1}(\rho)}\prod_{i=0}^{m-2} P^{v_i}_{T^*(v_{i+1}),\beta}(\eta_{i-1}<\infty)^s \left( \sum_{v\in G_1(v_{m-1})} P^{v_{m-1}}_{T^*(v),\beta}(\eta_{m-2}<\infty)^s \right) \right] \\
& = \mathrm{E}^{(N)}\left[ \sum_{v_{m-1}\in G_{m-1}(\rho)}\prod_{i=0}^{m-2} P^{v_i}_{T^*(v_{i+1}),\beta}(\eta_{i-1}<\infty)^s \right. \\
& \qquad \qquad \qquad \qquad \times \left. \mathrm{E}^{(N)}\left[\left. \sum_{v\in G_1(v_{m-1})} P^{v_{m-1}}_{T^*(v),\beta}(\eta_{m-2}<\infty)^s \right| T^*(v_{m-1}) \right]\right] \\
& = A_{N,\beta}(s)\mathrm{E}^{(N)}\left[ \sum_{v_{m-1}\in G_{m-1}(\rho)}\prod_{i=0}^{m-2} P^{v_i}_{T^*(v_{i+1}),\beta}(\eta_{i-1}<\infty)^s  \right] \\
& = A_{N,\beta}(s) \hat W_{m-1} ,   
\end{align*}
with 
\begin{align} \label{eq:defA}
A_{N,\beta}(s) = \mathrm{E}^{(N)}\left[ \sum_{v\in G_1(\rho)}P^{\rho}_{T^*(v),\beta}(\eta_{-1}<\infty)^s\right] = \mathrm{E}^{(N)}\left[ \sum_{i=1}^{d(\rho)} P^{\rho}_{T^*(w_i),\beta}(\eta_{-1}<\infty)^s\right] ,
\end{align}
where in the last expectation, we wrote $w_1,\dots w_{d(\rho)}$ for the children of $\rho$. We obtain recursively
\begin{align} \label{eq:Wrecursion}
W_m \leq \hat W_m \leq A_{N,\beta}(s)^m .
\end{align}
For $w_i$ as above, let $T_i$ be the subtree consisting of $w_i$ and all its descendants. The probability to hit $\rho^*$ on the reduced tree can be calculated as
\begin{align*}
P^{\rho}_{T^*(w_i),\beta}(\eta_{-1}<\infty) = \left( 1+\beta \sum_{1\leq j \leq d(\rho), j\neq i} e_\beta(T_j) \right)^{-1} ,
\end{align*}
with $e_\beta(T_j)$ as in \eqref{eq:defescapeprob}. Let us write $e_{\beta,j},j\geq 1$ for independent copies of $e_\beta(T_1)$. Using the exchangeability of the subtrees $T_j$ conditioned on $d(\rho)$, we obtain that \eqref{eq:defA} satisfies 
\begin{align*}
A_{N,\beta}(s) = \alpha N  \mathrm{E}^{(N)} \left[ \left( 1+\beta \sum_{j=1}^{N-1} e_{\beta,j} \right)^{-s}\right]
 + (1-\alpha)\sum_{k=1}^\infty  k p_k  \mathrm{E}^{(N)}\left[ \left( 1+\beta \sum_{j=1}^{k-1} e_{\beta,j} \right)^{-s}\right] .
\end{align*}
Choose $\delta>0$ such that $q=\mathrm{P}^{(N)}(e_{\beta,1}\geq \delta)>0$. The number of $e_{\beta,1},\dots ,e_{\beta,k-1}$ satisfying $ e_{\beta,j}\geq \delta$ is binomially distributed with parameters $k-1,q$. The probability that this number is smaller than $q(k-1)/2$ is then smaller than $e^{-c(k-1)}$ for some $c>0$. This implies
\begin{align*}
 k \mathrm{E}^{(N)}\left[ \left( 1+\beta \sum_{j=1}^{k-1} e_{\beta,j} \right)^{-s}\right] & \leq k e^{-c(k-1)} + k \left( 1+\tfrac{1}{2}\beta \delta q(k-1) \right)^{-s}  \\
 & \leq k e^{-c(k-1)} + k \left( 1+\tfrac{1}{2}\beta \delta q(k-1) \right)^{-2}  ,
\end{align*}
which tends to zero as $k \to \infty$. For $k>1$ fixed, 
\begin{align*}
k \mathrm{E}^{(N)}\left[ \left( 1+\beta \sum_{j=1}^{k-1} e_{\beta,j} \right)^{-s}\right] \xrightarrow[s\to \infty]{} 0
\end{align*}
by dominated convergence. Note that the contribution of a single child, corresponding to  $k=1$, does not vanish as $s\to \infty$. Again by dominated convergence, this implies 
\begin{align*}
A_{N,\beta}(s)  \xrightarrow[s\to \infty]{} (1-\alpha)p_1<1
\end{align*}
uniformly in $N$. In particular, we may choose $s\geq 2$ such that 
\begin{align*}
a_\beta :=  \sup_{N\geq N_0}A_{N,\beta}(s) <1 .
\end{align*}
With this choice, \eqref{eq:Wrecursion} gives $W_m\leq a_\beta^m$ and returning to \eqref{eq:regheightrecursion}, 
\begin{align*}
\bbP_\beta^{(N)} (H\geq m \mid R) \leq 
\left((1-\alpha) p_\ell \frac{1}{\beta \ell +1} \right)^{-1} a_\beta^{m/s} , 
\end{align*}
which decays exponentially as $m\to \infty$ uniformly in $N$. This yields the desired exponential moment for \eqref{eq:regheight}. 
\end{proof}

\begin{lem}\label{lem:momentuniform}
Suppose $p_0=0$. For any $\beta>0$ there exists $N_0\in \mathbb N$, such that for any $r\geq 1$  
\begin{align*}
\sup_{N\geq N_0} \bar{\mathbb{E}}^{(N)}_{\beta}[\tau_1^r]<\infty  . 
\end{align*}
\end{lem}

\begin{proof}
Let $N_0$ be sufficiently large, such that $r_0$, as defined in \eqref{eq:returnprobbound} is positive.  Recall that $\eta_m$ is the hitting time of generation $m$ and as in the proof of Lemma \ref{lem:moment-uniform2nd}, denote by $\theta_k$ the times at which a new vertex has been discovered, stopping when $\rho^*$ is visited. The arguments leading to \eqref{eq:thetatailbound} show that for any $q\geq 1$ there exists a constant $C_q$, such that
\begin{align*}
E_{T^*,\beta}[(\theta_k-\theta_{k-1})^q ] = E_{D,\beta}[(\theta_k-\theta_{k-1})^q ] 
 \leq C_qk^{2q+2} 
\end{align*} 
for any augmented tree $T^*$ or offspring sequence $D$, respectively, any $\beta\geq 0$ and any $k\geq 1$. Minkowski's inequality then gives
\begin{align} \label{eq:thetabound2}
E_{T^*,\beta}[\theta_k^q ] \leq C_qk^{3q+2} ,
\end{align}
with a modified constant $C_q$. 

We now argue similar to the proof of \cite[Lemma 3.6]{glatzel2024CLT}, which is based on \cite{Aidekon2010LargeDeviations}. 
The probability for a large regeneration time can be bounded as
\begin{align*}
\bar{\bbP}_\beta^{(N)}(\tau_1> n^7 ) 
& \leq \bar{\bbP}_\beta^{(N)}(\tau_1 >n^7 , \tau_1\geq \eta_{n} ) + \bar{\bbP}_\beta^{(N)}(\tau_1\geq n^7, \tau_1 <\eta_{n}) \\
& \leq \bar{\bbP}_\beta^{(N)}(|X_{\tau_1}| >n  ) + \bar{\bbP}_\beta^{(N)}(\eta_{n}>n^7 ) .
\end{align*}
The first probability decays exponentially as $n\to \infty$ uniformly in $N$ by Lemma \ref{lem:momentuniform-height}. The hitting time $\eta_n$ of generation $n$ is larger than $n^7$ when either more than $n^2$ different vertices have been visited until $\eta_n$ or it takes more than $n^7$ steps to visit $n^2$ distinct vertices. In the latter case, none of these vertices is $\rho^*$ when conditioning on $R^c$. Therefore, if $\eta_n$ is larger than $n^7$, either more than  $n^2$ different vertices have been visited until $\eta_n$ or $\theta_{n^2}>n^7$, such that 
\begin{align} \label{eq:regtimebound}
\bar{\bbP}_\beta^{(N)}(\eta_{n}>n^7 ) \leq \bar\bbP_\beta^{(N)}(|\{X_0,\dots ,X_{\eta_n}\} |>n^2) + \bar\bbP_\beta^{(N)}(\theta_{n^2}>n^7) .
\end{align} 
If $|\{X_0,\dots ,X_{\eta_n}\} |>n^2$, the random walk has to visit at least $n$ distinct vertices in one of the first $n$ generations, which implies $n$ returns to a parent. Following \cite[Section 4.4.1]{glatzel2024CLT}, we have 
\begin{align}
\bbP_\beta^{(N)}(|\{X_0,\dots ,X_{\eta_n}\} \cap G_k(\rho)|>n ) \leq \bbP_\beta^{(N)}(R)^n \leq (1-r_0)^n
\end{align}
by \eqref{eq:returnprobbound} for $1\leq k\leq n$. Therefore, by \eqref{eq:conditionallawbound} and a union bound,
\begin{align}
\bar{\bbP}_\beta^{(N)}(|\{X_0,\dots ,X_{\eta_n}\} |>n^2) \leq r_0^{-1} n(1-r_0)^n .
\end{align} 
For the second probability in \eqref{eq:regtimebound}, the bound \eqref{eq:conditionallawbound}, Markov's inequality and the moment bound in \eqref{eq:thetabound2} yield
\begin{align}
\bar\bbP_\beta^{(N)}(\theta_{n^2}>n^7) \leq r_0^{-1} \bbP_\beta^{(N)}(\theta_{n^2}>n^7) \leq r_0^{-1}C_q n^{4-q}, 
\end{align}
with $q\geq 1$ arbitrarily large. This implies the finite $r$-th moment of $\tau_1$, when $q$ is large enough. 
\end{proof}

\textbf{Statement on AI usage:}  {ChatGPT-6} has been utilized to check the mathematical arguments for correctness. The discussions with the AI fixed a gap in the author's first proof of Lemma \ref{lem:momentuniform} and modified the original reasoning in this proof. These discussions also lead to a sharper estimate in Theorem \ref{thm:speedmonotone}. Mathematical arguments which have been suggested by the AI have been reformulated and carefully checked by the author, who takes full responsibility for the manuscript.

\bibliographystyle{alpha}
\bibliography{bibGWT}

\bigskip

{\footnotesize
\noindent
TU Dortmund, \\
Fakult\"at f\"ur Mathematik, \\
Vogelpothsweg 87, 
44227 Dortmund, 
Germany, \\
jan.nagel@tu-dortmund.de
}

\end{document}